%% file: main.tex
\documentclass[12pt]{amsart}
\input{Preamble.tex}
\begin{document}
\input{Abstract.tex}
\maketitle
%\tableofcontents 
\input{Introduction.tex}
\input{CuspidalSupport.tex}
\input{HoweLusztigCorrespondence.tex}

\input{MinimalRepresentations.tex}
\input{UnipotentSupportRank.tex}
\input{WeakHarishChandra.tex}
\bibliography{DefenseAMSSbib}
\bibliographystyle{acm}
\end{document}

%% file: Preamble.tex
\usepackage[english]{babel} % For english language
\usepackage{amsfonts} 
\usepackage{amsmath} % For math symbols
\usepackage{amsthm} % For theorem environments
\usepackage{amssymb} % For more symbols (semi direct products)
\usepackage{mathrsfs}
\usepackage{mathtools}
\usepackage{tikz-cd}
\usepackage{rotating}
\usepackage{verbatim}

\title[Weak cuspidality and ...]{Weak cuspidality and the Howe correspondence}
\author{Jesua Epequin}
\date{}

\DeclareMathOperator{\End}{End}

\DeclareMathOperator{\U}{U}
\DeclareMathOperator{\Or}{O}
\DeclareMathOperator{\GL}{GL}
\DeclareMathOperator{\Sp}{Sp}
\DeclareMathOperator{\Irr}{Irr}
\DeclareMathOperator{\Ind}{Ind}

\DeclareMathOperator{\SO}{SO}
\DeclareMathOperator{\sgn}{sgn}

\DeclareMathOperator{\rk}{rk}

\DeclareMathOperator{\diag}{diag}

\DeclareMathOperator{\tr}{tr}
\DeclareMathOperator{\rank}{rank}

\DeclareFontEncoding{LS1}{}{}
\DeclareFontSubstitution{LS1}{stix}{m}{n}
\DeclareSymbolFont{symbols2}{LS1}{stixfrak} {m} {n}
\DeclareMathSymbol{\operp}{\mathbin}{symbols2}{"A8}

\tikzset{labr/.style={anchor=north, rotate=90, inner sep=1mm}}
\tikzset{labl/.style={anchor=south, rotate=90, inner sep=.5mm}}

\newtheorem{thm}{Theorem}

\newtheorem{defi}{Definition}

\newtheorem{lem}{Lemma}
\newtheorem{prop}{Proposition}
\newtheorem{cor}{Corollary}
\newtheorem{conj}{Conjecture}

\newtheorem*{thm*}{Theorem}

%% file: Abstract.tex
\begin{abstract}
We study the effect of the Howe correspondence on Harish-Chandra series for type I dual pairs over finite fields with odd characteristic. We define a bijection obtained from this correspondence, and enjoying the property of ``having minimal unipotent support''. Finally, we examine the interaction between the Howe correspondence and weak cuspidality.  %Finally, comparing this bijection to Howe and Gurevich's eta correspondence we deduce a reverse relation between the unipotent support and the rank of representations.
\end{abstract}

%% file: Introduction.tex
\section*{Introduction}
Let $\mathbb{F}_q$ be a finite field with $q$ elements and odd characteristic. A pair of reductive subgroups of $\Sp_{2n}(q)$, where each one is the centralizer of the other, is called \emph{reductive dual pair}. We focus our attention on \emph{irreducible} dual pairs (cf. \cite{Kudla}). One such pair can be \emph{linear} $(\GL_m(q),\GL_{m'}(q))$, \emph{unitary} $(\U_m(q),\U_{m'}(q))$ or \emph{symplectic-orthogonal} $(\Sp_{2m}(q),\Or_{m'}(q))$, with $n=mm'$ in all cases. The last two are also called \emph{type I} dual pairs, and the groups belonging to them are called \emph{type I groups}.

For a reductive dual pair $(G_m,G'_{m'})$, Roger Howe defined the a correspondence $\Theta_{m,m'} : \mathscr{R}(G_m)\rightarrow\mathscr{R}(G'_{m'})$ between their categories of complex representations. Known as \emph{Howe correspondance}, it arises from the restriction to $G_m\times G'_{m'}$ of the \emph{Weil representation} $\omega$ of the symplectic group $\Sp_{2n}(q)$ (cf. \cite{Howe}). 

Let $\mathbf{G}$ be a connected reductive group defined over $\mathbb{F}_q$, and $\mathbf{G}^*$ its dual group. Denote by $G$ and $G^*$ their groups of rational points. In \cite{Lusztig3} Lusztig defined a partition of the set $\mathscr{E}(G)$ of irreducible representations of $G$ into \emph{Lusztig series} $\mathscr{E}(G,(s))$. These series are parametrized by rational conjugacy classes of semisimple elements $s$ of $G^*$. The elements of $\mathscr{E}(G,(1))$ are called \emph{unipotent representations} of $G$.

In general, the Howe correspondence is not compatible with unipotent representations. Therefore, we make use of a similar correspondance $\Theta^\flat_{m,m'}:\mathscr{R}(G_m)\rightarrow\mathscr{R}(G'_{m'})$ arising from a Weil representation $\omega^\flat$ introduced by G\'erardin in \cite{Gerardin}. This correspondence preserves unipotent representations (cf. Proposition 2.3 in \cite{AMR}). Since in this paper we only work with this modified Howe correspondence we will drop the superscript and denote $\Theta^\flat_{m,m'}$, and $\omega^\flat$ by $\Theta_{m,m'}$, and $\omega$ respectively.

In \cite{Lusztig4} Lusztig found that, for every irreducible representation $\pi$ of a connected reductive group $\mathbf{G}$, there is a unique rational unipotent class $\mathscr{O}_{\pi}$ in $\mathbf{G}$ which has the property that $\sum_{x\in \mathscr{O}_{\pi}(q)}\pi(x)$ is non trivial, and that has maximal dimension among classes with this property. This class is called the \emph{unipotent support} of $\pi$. Such classes are ordered by the relation given by $\mathcal{O}'\preceq\mathcal{O}$, if and only if $\mathcal{O}'\subset\overline{\mathcal{O}}$, referred to as the \emph{closure order}. 

 In \cite{Epequin} we defined a bijective mapping $\theta_{G,G'}:\mathscr{E}(G,1)\rightarrow\mathscr{E}(G',1)$ between the unipotent series of the members of a type I dual pair $(G,G')$; in such a way that, for a unipotent representation $\pi$ of $G$

-- The representation $\theta_{G,G'}(\pi)$ occurs in $\Theta_{G,G'}(\pi)$.

-- If $\pi'$ belongs to $\Theta_{G,G'}(\pi)$, then $\mathcal{O}_{\theta(\pi)}\preceq\mathcal{O}_{\pi'}$ 

 The last item above asserts that $\theta_{G,G'}(\pi)$ has the smallest unipotent support among irreducible representations in $\Theta_{G,G'}(\pi)$, it is in this sense minimal. 
 
We want to extend this definition to the whole set of irreducible representations. Naturally, any attempt to do this must make use of the \emph{Lusztig bijection}. This bijection, known as well as the \emph{Lusztig correspondence} is a one-to-one map between the series $\mathscr{E}(G,(s))$ and the series $\mathscr{E}(C_{G^*}(s),(1))$, of unipotent representations of the centralizer $C_{G^*}(s)$. For classical groups, this centralizer can be expressed as a product of smaller reductive groups.
 
For instance, when $G$ is a unitary group 
$$
C_{G^*}(s) \simeq G_\#\times G_{(1)}
$$
where $G_\#$ is a product of linear or unitary groups, and $G_{(1)}$ is a unitary group. We obtain in this way a modified Lusztig correspondence $\Xi_s$ sending $\pi \in \mathscr{E}(G,(s))$ to $\pi_\#\otimes\pi_{(1)}\in \mathscr{E}(G_\# \times G_{(1)},1)$. For a unitary dual pair $(G,G')$ this bijection fits in the commutative diagram 

\begin{center} 
\begin{tikzcd}
 \mathscr{E}(G,(s)) \arrow[d, "\Theta_{G,G'}"] \arrow[r, "\Xi_{s}", "\sim" swap] 
                                                &  \mathscr{E}(G_\#\times G_{(1)},1) \arrow[d, "\iota\hspace{0.05cm}\otimes\hspace{0.05cm}\Theta_{G_{(1)},G'_{(1)}}"] \\
\mathscr{R}(G',(s')) \arrow[r,  "\Xi_{s'}", "\sim" swap] 
 												   & \mathscr{R}(G_\#\times G'_{(1)},1).
 \end{tikzcd}
\end{center}

This motivates us to define $\theta_{G,G'}(\pi)$, for $\pi$ in $\mathscr{E}(G)$, to be the unique irreducible representation of $G'$ such that 
$$
\Xi_{s'}(\theta_{G,G'}(\pi))=\tilde{\pi}_\#\times \theta_{G_{(1)},G'_{(1)}}(\pi_{(1)}).
$$ 
We also define such a mapping $\theta_{G,G'}:\mathscr{E}(G)\rightarrow\mathscr{E}(G')$, for symplectic-orthogonal pairs. As we could expect, this bijection also selects representations with smallest unipotent supports. This is our first result (Theorem \ref{IncreasingLusztig}).

\begin{thm*}
Let $\pi$, $\tau$ belong to $\mathscr{E}(G,(s))$, and $\pi_u$, $\tau_u$ denote the corresponding unipotent representations of $C_{G^*}(s)$. If $\mathcal{O}_{\pi_u}\preceq\mathcal{O}_{\tau_u}$ then $\mathcal{O}_{\pi}\preceq\mathcal{O}_{\tau}$. In particular, for $\pi$ in $\mathscr{E}(G)$, the subrepresentation $\theta_{G,G'}(\pi)$ of $\Theta_{G,G'}(\pi)$ has the smallest unipotent support.
\end{thm*}

In a recent paper \cite{Howe-Gurevich}, Howe and Gurevich presented the notion of rank for representations of finite symplectic groups (see Section \ref{SupportRank}). This conduced to the introduction of the \emph{eta} correspondence. Consider a dual pair $(G,G')$ formed by one orthogonal and one symplectic group; and suppose the pair is in the stable range.  Howe and Gurevich show that for $\rho$ in $\mathscr{E}(G)$, there is a unique irreducible representation in $\Theta_{G,G'}(\rho)$ of maximal rank, it is denoted by $\eta(\rho)$.
 
 The correspondences $\theta$ and $\eta$ are defined in different ways. The former chooses a subrepresentation of $\Theta$ with smallest unipotent support, whereas the latter selects one with greatest rank. In \cite{Pan3}, Shu-Yen Pan shows that these two agree on their common domain of definition (the stable range), i.e. among the irreducible constituents of $\Theta(\pi)$, the representation with the smallest unipotent support is the one having the greatest rank. This points out to an inverse relation between these two features. We discuss this in Section \ref{SupportRank}. %This inverse relation holds in general, as stated in our third and final theorem (see Theorem \ref{InverseRelation}).
 
%\begin{thm*}
%Let $\pi$ and $\tau$ be two irreducible representations of $G$. If $\mathcal{O}_\pi \preceq\mathcal{O}_\tau$ then $\rk(\tau)\leq \rk(\pi)$.
%\end{thm*}

In \cite{GHJ}, Gerber, Hiss and jacon introduced the notion of \emph{weak cuspidality} for modular representations in non-defining (including zero) characteristic. The spirit of the definition is the same as for cuspidal representations: the vanishing of parabolic restriction functors. For weakly cuspidal representations we nonetheless restrict our attention to certain kind of parabolic subgroups, called \emph{pure} by the authors. This yields a weak Harish-Chandra theory that refines the usual one. Moreover, for representations of unitary groups in prime characteristic, this new definition provides a natural partition of the set of unipotent representations (see Proposition \ref{CoreWeaKHC}), in the same way usual series do for ordinary unipotent representations.

A natural question arises: how does the Howe correspondence behaves with respect to this weak Harish-Chandra theory? We provide an answer to this question for ordinary representations. We first prove that the Howe correspondence respects weak cuspidality in case of first occurrence. This is our second result.

\begin{thm*}
Let $(G_m,G'_{m'})$ be a type I dual pair, and $\pi$ be an irreducible weakly cuspidal representation of $G_m$, and let $m'(\pi)$ be its first occurrence index. 

1. If $m'<m'(\pi)$, then $\Theta_{m,m'}(\pi)$ is empty

2. The representation $\Theta_{m,m'(\pi)}(\pi)$ is irreducible and weakly cuspidal

3.  If $m'>m'(\pi)$, then none of the constituents of $\Theta_{m,m'}(\pi)$ is weakly cuspidal 
\end{thm*}
 
 The proof of this result can divided in two independent parts: existence and uniqueness, each making use of different tools. The former relies on the computation of coinvariants of the Weil representation presented in Section \ref{HoweCorrespondanceCuspidalSupport}. The latter, on the study of orbits and stabilizers for the action of a block diagonal subgroup of a type I group on the set of maximal isotropic subspaces of the underlying module of this group. We need to mention that the proof is inspired by the equivalent result for cuspidal representations, but that it is more refined. Indeed, for cuspidal representations, the uniqueness proof used the fact that the stabilizers above are contained in a direct product of parabolic subgroups. For weakly cuspidal representations this is not enough, we had to calculate this stabilizers explicitly. In order to do so we had first to find explicit representatives for the orbits (which is also not needed in the cuspidal setting).
 
  From the last theorem above, the work required to establish the agreement between the Howe correspondence and weak cuspidal support is very much the same required in the cuspidal setting.

%% file: CuspidalSupport.tex
\section{Howe correspondance and cuspidal support}\label{HoweCorrespondanceCuspidalSupport}
Theorem 3.7 of \cite{AMR} states that the Howe correspondence is compatible with unipotent Harish-Chandra series. In this section, we generalize this theorem to arbitrary Harish-Chandra series. The proof follows that of Theorem 2.5 in \cite{Kudla}. 
 
We fix two Witt towers, $\mathscr{T}$ and $\mathscr{T}'$, such that $(G_m,G'_{m'})$ is a type I dual pair for any $G_m\in\mathscr{T}$ and $G'_{m'}\in\mathscr{T}'$. 

Let $D$ be a field equal to $\mathbb{F}_q$ when the dual pair is symplectic-orthogonal, and equal to $\mathbb{F}_{q^2}$ when the pair is unitary. Let $W_m$ be the underlying $D$-vector space of $G_m$. Let $P_k$ be the stabilizer  in $G_m$ of the totally isotropic subspace of $W_m$, spanned by the $k$ first vectors of a hermitian base, $k \leq m$. Denote by $N_k$ its unipotent radical,  $\GL_k=\GL(D)$ and $M_k=\GL_k\times G_{m-k}$ the standard Levi subgroup of $P_k$. Denote by $\GL'_{k'}$, $P'_{k'}$, $N'_{k'}$, and $M'_{k'}$ the analogous groups for $G'_{m'}$. 
Finally, denote by $\mathbf{R}_G$, the natural representation of the group $G\times G$ on the space $\mathscr{S}(G)$. This representation is isomorphic to the one obtained by inducing the trivial representation of $G$ to $G\times G$ (diagonal inclusion). It decomposes as 
$$
\mathbf{R}_G = \sum_{\pi\in\mathscr{E}(G)} \pi\otimes \tilde{\pi},
$$
where $\tilde{\pi}$ denotes the contragredient representation of $\pi$. 

\begin{prop}\label{coinv-submod}\cite[Corollary 1]{Epequin2} 
Let $\prescript{*}{}{R}_k\otimes 1$ and $1\otimes\prescript{*}{}{R}'_{k'}$  be the parabolic restriction functor from $G_m\times G'_{m'}$ to $M_k\times G'_{m'}$ and $G_m\times M'_{k'}$ respectively.

\emph{a)} The representation $(\prescript{*}{}{R}_k\otimes 1)(\omega_{m,m'})$ decomposes as :
$$
\bigoplus_{i=0}^{\min\{k,m'\}} R_{\GL_{k-i}\times \GL_i \times G_{m-k}\times M'_i}^{M_k\times G'_{m'}} 1_{\GL_{k-i}}\otimes \mathbf{R}_{\GL_{i}}\otimes\omega_{m-k,m'-i}.
$$
 \emph{b)} Likewise, the representation $(1\otimes\prescript{*}{}{R}'_{k'})(\omega_{m,m'})$ decomposes as :
$$
\bigoplus_{i=0}^{\min\{k',m\}} R_{M_i\times \GL_{k'-i}\times\GL_i\times G'_{m'-k'}}^{G_m\times M'_{k'}} 1_{\GL_{k'-i}}\otimes\mathbf{R}_{\GL_{i}}\otimes\omega_{m-i,m'-k'}.
$$ 
\end{prop}

Let $G_m$ be a type I group in the Witt tower $\mathscr{T}$. The set of standard Levi subgroups of $G_m$ can be parametrized by sequences $\mathbf{t}=(t_1,\ldots,t_r)$, such that $|\mathbf{t}|= \sum_{i=1}^r t_i$ is not greater than $m$. The corresponding Levi subgroup is equal to $\GL_{t_1}\times \cdots \times \GL_{t_r} \times G_{m-|\mathbf{t}|}$. 

For this Levi subgroup, we denote parabolic induction by  $R_{\mathbf{t}}$ and parabolic restriction by $\prescript{*}{}{R_{\mathbf{t}}}$. The Harish-Chandra series corresponding to the representation $\sigma_1\otimes\cdots\otimes\sigma_r\otimes\varphi$ of this Levi is denoted by $\mathscr{E}(G_m,\boldsymbol{\sigma}\otimes\varphi)_\mathbf{t}$, where $\boldsymbol{\sigma}=\sigma_1\otimes\cdots\otimes\sigma_r$. Finally, for a cuspidal representation $\pi$ of $G_m$, we denote by $\pi'$ (resp. $m'(\pi)$) its first occurrence (resp. first occurrence index) \cite{AM}. 

Proposition \ref{coinv-submod} is a key result in the proof of the following. 
 
\begin{thm}\label{HoweHarish-Chandra}\cite[Theorem 3]{Epequin2}
The image of $\mathscr{E}(G_m,\boldsymbol{\sigma}\otimes\varphi)_\mathbf{t}$ by the correspondence $\Theta_{m,m'}$ is spanned by a single series $\mathscr{E}(G'_{m'},\boldsymbol{\sigma}'\otimes\varphi')_{\mathbf{t}'}$ whenever $m'\geq m'(\varphi)$ and is zero otherwise. In the first case, if $|\mathbf{t}'|\geq |\mathbf{t}|$ then $\boldsymbol{\sigma}'=\boldsymbol{\sigma}\otimes 1$ and $\mathbf{t}'=\mathbf{t}\cup 1^d$; if $|\mathbf{t}'| < |\mathbf{t}|$ then $\boldsymbol{\sigma}=\boldsymbol{\sigma}'\otimes 1$ and $\mathbf{t}=\mathbf{t}'\cup 1^d$, where $d=\vert\vert\mathbf{t}'\vert-\lvert\mathbf{t}\vert\rvert$.
\end{thm}

In \cite{Epequin2} we showed how this result implies Theorem 3.7 in \cite{AMR}. The latter basically says that the Howe correspondence preserves unipotent Harish-Chandra series. We also used Theorem \ref{HoweHarish-Chandra} to obtain the bijective correspondence $\theta$ we present in Section \ref{SectionExtremalRepresentations}. However, this was not necesary, we can define $\theta$ using only the Lusztig correspondence (as done below).

\begin{comment}
If we ask for the representation $\boldsymbol{\sigma}\otimes\varphi$ of $\GL_{\mathbf{t}}\times G_l$ to be also unipotent, then $\mathbf{t}$ becomes $\mathbf{t}=(1^{m-l})$, the representation $\boldsymbol{\sigma}$ becomes the trivial representation of the torus $\GL_1^{m-l}$ of diagonal matrices, and $\boldsymbol{\sigma}'$ becomes the trivial representation of the torus $\GL_1^{m'-m'(\varphi)}$. Therefore, as a particular case of Theorem \ref{HoweHarish-Chandra} we get Theorem 3.7 in \cite{AMR}. Let $\mathscr{E}(G_m)_\varphi$ denote the series $\mathscr{E}(G_m,1\otimes\varphi)_{1^{m-l}}$.

\begin{cor}\label{HoweUnipotentHC}
The Howe correspondence $\Theta_{m,m'}$ sends representations in the series $\mathscr{E}(G_m)_\varphi$ to representations spanned by $\mathscr{E}(G'_{m'})_{\varphi'}$ whenever $m'\geq m'(\varphi)$, and to zero otherwise. Moreover, the representation $\varphi$ is unipotent if and only if the same holds for $\varphi'$. 
\end{cor}
\end{comment}

%% file: HoweLusztigCorrespondence.tex
\section{Howe correspondence and Lusztig correspondence}\label{HoweLusztigCorrespondence}
The purpose of this section is to see the effect of the Lusztig correspondence on the Howe correspondence for type I dual pairs.

\subsection{Centralizers of rational semisimple elements}\label{CentralizersSection}
Let $\mathbf{G}$ be a reductive group defined over $\mathbb{F}_q$, and  $C_\textbf{G}(x)$ be the centralizer of a rational element $x$ in $\textbf{G}$. Denote by $G$ and $C_G(x)$ their groups of rational elements. 

Assume that $\mathbf{G}$ is also connected. In Proposition 5.1 of \cite{Lusztig}, Lusztig found a bijection
\begin{align}\label{LusztigBijection}
 \mathfrak{L}_s : \mathscr{E}(G,(s)) \simeq \mathscr{E}(C_{G^*}(s),(1)).
\end{align}
where $s$ is a rational semisimple element of $\textbf{G}^*$. Aubert, Michel and Rouquier extended this bijection to even orthogonal groups (Proposition 1.7 of \cite{AMR}). It is known as the \emph{Lusztig bijection} or \emph{Lusztig correspondence}. Taking $s=1$ yields a bijection between the series of unipotent representations of $G$ with that of its dual :
\begin{align}\label{UnipotentLusztigBijection}
\mathfrak{L}_1 : \mathscr{E}(G,1) \simeq \mathscr{E}(G^*,1),
\end{align}
We can also extend (\ref{LusztigBijection}) by linearity in order to obtain an isometry between the categories $\mathscr{R}(G,(s))$ and $\mathscr{R}(C_{G^*}(s),(1))$ spanned by the Lusztig series $\mathscr{E}(G,(s))$ and $\mathscr{E}(C_{G^*}(s),(1))$ respectively. 

Following Section 1.B in \cite{AMR}, let $\mathbf{G}$ be a classical group of rank $n$, and $\mathbf{T}_n=\mathbb{F}_q^n$. Let $s$ be a rational semisimple element of its Langlands dual $\mathbf{G^*}$. By definition, a rational semisimple element is conjugate to an element $(\lambda_1,\ldots,\lambda_n)$ in $\mathbf{T}_n$. Let $\nu_\lambda(s)$ the number of times $\lambda$ appears in this list. There is a decomposition
$$
C_\mathbf{G^*}(s)=\prod\mathbf{G}_{[\lambda]}(s),
$$
where $[\lambda]$ is the orbit of $\lambda$ by the action of the Frobenius endomorphism, intersected with $\{\lambda_1,\ldots,\lambda_n\}$. Each $\mathbf{G}_{[\lambda]}(s)$ is a reductive quasi-simple group of rank $|[\lambda]|\nu_\lambda(s)$. Moreover, if $\lambda\neq \pm 1$, then $\mathbf{G}_{[\lambda]}(s)$ is a unitary or general linear group (possibly over some finite extension of $\mathbb{F}_q$). Additionally :

\noindent (1) If $\mathbf{G}=\boldsymbol{\GL}_n$ is unitary, then $\mathbf{G}_{[\pm 1]}(s)$ is a unitary group.

\noindent (2) If $\mathbf{G}=\boldsymbol{\SO}_{2n+1}$, then $\mathbf{G}_{[-1]}(s)=\boldsymbol{\Or}_{2\nu_{-1}(s)}$, and $\mathbf{G}_{[1]}(s)=\boldsymbol{\SO}_{2\nu_1(s)+1}$.

\noindent (3) If $\mathbf{G}=\boldsymbol{\Or}_{2n}$, then $\mathbf{G}_{[\pm 1]}(s)=\boldsymbol{\Or}_{2\nu_{\pm 1}(s)}$.

In all cases we see that $\mathbf{G}_{[1]}(s)$ is a group of the same kind as $\mathbf{G}$, but of smaller rank. %We can replace $\mathbf{G}$ by any rational subgroup $\mathbf{H}$ containing $s$ and define $\mathbf{H}_{[\lambda]}(s)$ in a similar manner.

\begin{comment}
\begin{lem}\label{LeviCentralizer}
Let $\mathbf{L}$ be a Levi complement of the parabolic subgroup $\mathbf{P}$ of $\mathbf{G}$, both groups being rational. If $s$ is a semisimple element of $L$, then $\mathbf{L}_{[\lambda]}(s)$ is a Levi subgroup of the parabolic $\mathbf{P}_{[\lambda]}(s)$ of $\mathbf{G}_{[\lambda]}(s)$. Moreover, these groups are also rational.
\end{lem}  
\begin{proof}
Since $C_\textbf{G}(s)$ has maximal rank in $\mathbf{G}$, Proposition 2.2 in \cite{Digne-Michel} implies that $C_\textbf{P}(s) = P \cap C_\textbf{G}(s)$ is a parabolic subgroup of $C_\textbf{G}(s)$. 

The subgroup $\mathbf{P}_{[\lambda]}(s)$ of $\mathbf{G}_{[\lambda]}(s)$ is parabolic because the quotient $\mathbf{G}_{[\lambda]}(s)/\mathbf{P}_{[\lambda]}(s)$ is projective. Indeed, this quotient is closed in the projective variety $C_\textbf{G}(s)/C_\textbf{P}(s)$.  

Proposition 2.2 asserts also that $C_\textbf{L}(s) = L \cap C_\textbf{G}(s)$ is a Levi subgroup of $C_\textbf{P}(s)$, whence $\mathbf{L}_{[\lambda]}(s)$ is a Levi subgroup of $\mathbf{P}_{[\lambda]}(s)$. The final assertion is obvious.
\end{proof}
\end{comment}

\subsection{Weil representation and Lusztig correspondence}
Consider a type I dual pair $(G,G')$. Let $m$ (resp. $m'$) be the Witt index of $G$ (resp. $G'$). According to Proposition 2.3 in \cite{AMR}, if $s$ is a rational semisimple element in $\mathbf{G}^*$, then there is a rational semisimple element $s'$ in $\mathbf{G'}^*$, such that the Howe correspondence relates $\mathscr{E}(G,(s))$ and $\mathscr{E}(G',(s'))$. Moreover, in this case, $s'=(s,1)$ if $m\leq m'$, and $s=(s',1)$ otherwise. In particular, there is some $l\leq\min(m,m')$, and $t$ in $\mathbf{T}_l$ with eigenvalues different from $1$, such that $s=(t,1)$, and $s'=(t,1)$. Let $\omega_{G,G',t}$ denote the projection of the Weil representation $\omega_{G,G'}$ onto $\mathscr{R}(G,(s))\otimes\mathscr{R}(G',(s'))$, and $\mathbf{T}_{l,0}$ the subset of $\mathbf{T}_l$ whose elements have all their eigenvalues different from $1$. Proposition 2.4 in \cite{AMR} asserts that
$$
\omega_{G,G'}=\bigoplus_{l=0}^{\min(m,m')}\bigoplus_{t\in\mathbf{T}_{l,0}}\omega_{G,G',t}.
$$
We now endeavour to study the effect of the Lusztig correspondence on the Weil representation $\omega_{G,G',t}$. We treat unitary and symplectic-orthogonal groups independently.

\subsubsection{Unitary pairs}
Suppose that $\mathbf{G}$ is a unitary group. Let $s$ be a rational semisimple element in $\mathbf{G}^*$, let $\mathbf{G}_\#$ denote the product of $\mathbf{G}_{[\lambda]}(s)$ for $\lambda\neq 1$, and $\mathbf{G}_{(1)}$ be the Langlands dual of $\mathbf{G}_{[1]}(s)$. The groups of rational elements of $\mathbf{G}_\#$ and $\mathbf{G}_{(1)}$ will be denoted by $G_\#$ and $G_{(1)}$ respectively. Considering the decomposition of centralizers discussed above, and since by (\ref{UnipotentLusztigBijection}) the unipotent Lusztig series of $\mathbf{G}_{(1)}$ and $\mathbf{G}_{[1]}$ can be identified, we obtain a modified Lusztig bijection 

\begin{align}\label{UnitaryPanBijection}
\Xi_s : \mathscr{E}(G,(s))\simeq \mathscr{E}(G_\#,1)\times \mathscr{E}(G_{(1)},1).
\end{align}

%This bijection is easily extended to rational Levi subgroups $\mathbf{L}$ of $\mathbf{G}$.

%If $\mathbf{G}=\mathbf{G}_m$ has Witt index $m$, and we let $l$ denote the rank of $G_\#(s)$, then the group $G^{(1)}(s)$ can be identified with the group $G_{m-l}$. In this case, 

For $\pi$ in $\mathscr{E}(G,(s))$ we will denote by $\pi_\#$ and $\pi_{(1)}$ the (unipotent) representations of $G_\#$ and $G_{(1)}$ such that
$$
\Xi_s(\pi)=\pi_\#\otimes \pi_{(1)}
$$ 
Let $(G,G')$ be a unitary dual pair, and $s$ (resp. $s'$) be a semisimple element of $G^*$ (resp. $G^{'*}$). 
\begin{prop}\label{UnitaryPanGroups}
The groups $G_\#$ and $G'_\#$ are isomorphic. Moreover, the pair $(G_{(1)},G'_{(1)})$ can be identified with a unitary dual pair.
\end{prop}
\begin{proof}
Both assertions in the statement above follow from the explicit decomposition of centralizers given in the Section \ref{CentralizersSection}. The isomorphism between  $G_\#$ and $G'_\#$ is a consequence of the fact that $s$ and $s'$ have the same eigenvalues different from 1 (with same multiplicities). Concerning the second assertion, it suffices to state that $G_{[1]}$ and $G'_{[1]}$ are unitary groups.
\end{proof}

Finally, the Weil representation $\omega_{G,G',t}$ can be described in terms of a correspondence between unipotent characters, defined either by $\mathbf{R}_{G_\#,1}$, or by the unipotent projection of the Weil representation of the smaller unitary dual pair $(G_{(1)},G'_{(1)})$.

\begin{thm}\label{UnitaryReductionUnipotentCase}\cite[Th\'eor\`eme 2.6]{AMR}
Let $t$ belong to $\mathbf{T}_{l,0}$. For a linear or unitary pair $(G,G')$, the representation $\omega_{G,G',t}$ is the image by the Lusztig correspondance
$$
\mathscr{E}(G\times G',(s\times s')) \simeq \mathscr{E}(C_{G^*}(s)\times C_{G^{'*}}(s'),1),
$$
of the representation
$$
\mathbf{R}_{G_\#,1}\otimes \omega_{G_{(1)},G'_{(1)},1}.
$$
\end{thm}

\subsubsection{Symplectic-orthogonal pairs}

Suppose that $\mathbf{G}$ is a symplectic, or an even-orthogonal group. Let $s$ be a rational semisimple element in $\mathbf{G}^*$, let $\mathbf{G}_\#$ denote the product of $\mathbf{G}_{[\lambda]}(s)$ for $\lambda\neq \pm 1$, let $\mathbf{G}_{(-1)}=\mathbf{G}_{[-1]}(s)$, and $\mathbf{G}_{(1)}$ be the Langlands dual of $\mathbf{G}_{[1]}(s)$. Again, considering the decomposition of centralizers discussed above, and that since by (\ref{UnipotentLusztigBijection}) the unipotent Lusztig series of $\mathbf{G}_{(1)}$ and $\mathbf{G}_{[1]}(s)$ can be identified, we obtain a modified Lusztig bijection 
\begin{align}\label{Symplectic-OrthogonalPanBijection}
\Xi_s : \mathscr{E}(G,(s))\simeq \mathscr{E}(G_\#,1)\times\mathscr{E}(G_{(-1)},1)\times\mathscr{E}(G_{(1)},1).
\end{align}
%Again, we can easily extended this bijection to rational Levi subgroups $\mathbf{L}$ of $\mathbf{G}$.

%and identify $\mathbf{G}^*$ with $\mathbf{G}$. Let $s$ be a rational semisimple element, $\mathbf{G}_\#(s)$ denote the product of $\mathbf{G}_{[\lambda]}(s)$ for $\lambda\neq 1$, and let $l$ be its rank. Let $G_\#(s)$ and $G_{[1]}(s)$ denote the group of rational elements of  $\mathbf{G}_\#(s)$ and $\mathbf{G}_{[1]}(s)$ respectively. If $\mathbf{G}=\mathbf{G}_m$ has Witt index $m$, then the group $G_{[1]}(s)$ can be identified with the unitary group $G_{m-l}$. The decomposition of the centralizer of $s$ (found in the previous section) together with the Lusztig correspondence yield a one-to-one correspondence 
%\begin{align}\label{BijectionPan}
%\Xi^G_s : \mathscr{E}(G_m,(s))\simeq \mathscr{E}(G_\#(s),(1))\times \mathscr{E}(G_{m-l},(1)).
%\end{align}
%If $\mathbf{G}=\mathbf{G}_m$ has Witt index $m$, and we let $l$ denote the sum of the ranks of $G_\#(s)$ and $G_{[-1]}(s)$, then the group $G_{(1)}(s)$ can be identified with the group $G_{m-l}$. 

For $\pi$ in $\mathscr{E}(G,(s))$ we will denote by $\pi_\#$, $\pi_{(-1)}$ and $\pi_{(1)}$ the (unipotent) representations of $G_\#$, $G_{(-1)}$ and $G_{(1)}$ such that 
$$
\Xi_s(\pi)=\pi_\#\otimes \pi_{(-1)}\otimes\pi_{(1)}.
$$

Let $(G,G')$ be a symplectic-orthogonal dual pair, and $s$ (resp. $s'$) be a semisimple element of $G^*$ (resp. $G^{'*}$).  %According to Proposition 2.3 in \cite{AMR}, if $s_m$ is a semisimple element in $G_m$, then there is a semisimple element $s'_{m'}$ in $G'_{m'}$, such that the Howe correspondence sends $\mathscr{E}(G_m,(s_m))$ to $\mathscr{R}(G'_{m'},(s'_{m'}))$. Moreover, in this case there is some $l\leq\min(m,m')$, and some $s\in \mathbf{T}_l$ with eigenvalues different from $1$, such that $s_m=(s,1)$, and $s'_{m'}=(s,1)$.

\begin{prop}\label{Symplectic-OrthogonalPanGroups}
The groups $G_\#$ and $G_{(-1)}$ are isomorphic to $G'_\#$ and $G'_{(-1)}$ respectively. Moreover, the pair $(G_{(1)},G'_{(1)})$ is a symplectic orthogonal dual pair. %Hence, if $l$ denotes their common rank then $((G_m)^{(1)}(s),(G'_{m'})^{(1)}(s'))$ can be identified with the dual pair $(G_{m-l},G'_{m'-l})$.
\end{prop}
\begin{proof}
The proof of the first assertion  is the same as that of Proposition \ref{UnitaryPanGroups}, with the difference that the groups $\boldsymbol{G}_{[-1]}(s)$ and $\boldsymbol{G}'_{[-1]}(s')$ are in this case isomorphic to $\boldsymbol{\Or}_{2\nu_{-1}(t)}$. For the second assertion, since $\mathbf{G}$ is symplectic, the group $\mathbf{G}_{[1]}(s)$ is special odd orthogonal and hence its dual is again symplectic. Likewise, the group $\mathbf{G}'_{[1]}(s')$ is even orthogonal, and so is its dual.
%The isomorphism between  $(G_m)_\#(s)$ and $(G'_{m'})_\#(s')$ comes from the explicit decomposition of centralizers given in the Section \ref{CentralizersSection}, and the fact that $s$ and $s'$ have the same eigenvalues different from 1 (with same multiplicities), this comes from the last sentence in the previous paragraph. The second assertion follows from the discussion above.
\end{proof}

As for unitary pairs, we now describe the Weil representation $\omega_{G,G',t}$ for symplectic-orthogonal pairs.

\begin{thm}\label{Symplectic-OrthogonalReductionUnipotentCase}\cite[Theorem 6.9 and Remark 6.10]{Pan5}
Let $t$ belong to $\mathbf{T}_{l,0}$. For a symplectic orthogonal dual pair $(G,G')$ the representation $\omega_{G,G',t}$ is the image by the Lusztig correspondence
$$
\mathscr{E}(G\times G',(s\times s')) \simeq \mathscr{E}(C_{G^*}(s)\times C_{G^{'*}}(s'),1).
$$
of the representation
$$
\mathbf{R}_{G_\#,1}\otimes \mathbf{R}_{G_{(-1)},1}\otimes\omega_{G_{(1)},G'_{(1)},1}.
$$
\end{thm}

%% file: MinimalRepresentations.tex
\section{Minimal representations}\label{SectionExtremalRepresentations}

Let $\mathbf{G}$ be a reductive group defined over $\mathbb{F}_q$, and $\mathbf{P}=\mathbf{L}\mathbf{U}$ be a Levi decomposition of the rational parabolic subgroup $\mathbf{P}$. For a cuspidal representation $\rho$ of $L$ set
$$
W_\mathbf{G}(\rho) = \{x\in N_{G}(\mathbf{L})/L : {}^x\rho=\rho\}.
$$
By Corollary 5.4 in \cite{HL} and Corollary 2 in \cite{Geck}, there is an isomorphism 
\begin{align}\label{Howlett-Lehrer}
\End_G(R_\mathbf{L}^\mathbf{G}(\rho)) \simeq \mathbb{C}[W_\mathbf{G}(\rho)].
\end{align}
In particular, irreducible representations in the Harish-Chandra series $\mathscr{E}(G,\rho)_L$ are indexed by irreducible representations of $W_\mathbf{G}(\rho)$.

When $G$ is a type I group and $\rho$ is a cuspidal unipotent representation, the group $ W_\mathbf{G}(\rho)$ above is a type B Weyl group. It is known that irreducible representations of these groups are parametrized by bipartitions. We denote by $\rho_{\mu,\lambda}$ the representation in $\mathscr{E}(W_r)$ corresponding to the bipartition $(\mu,\lambda)$ of $r$.

 For every irreducible representation $\pi$ of a connected reductive group $G$, there is a unique rational unipotent class $\mathscr{O}_{\pi}$ in $G$ which has the property that $\sum_{x\in \mathscr{O}_{\pi}(q)}\pi(x)$ is non trivial, and that has maximal dimension among classes with this property. This class, introduced by Lusztig in \cite{Lusztig4} is called the \emph{unipotent support} of $\pi$.
 
We now introduce a partial order on the set of unipotent conjugacy classes. It is crucial for results below. 

\begin{defi} 
The relation on the set of unipotent conjugacy classes of $G$, given by $\mathcal{O}'\preceq\mathcal{O}$, if and only if $\mathcal{O}'\subset\overline{\mathcal{O}}$, is a partial order. We refer to it as the \emph{closure order}.
\end{defi} 
 
Let $(G,G')$ be a type I dual pair. In \cite{Epequin} we defined a bijective correspondence $\theta_{G,G'}:\mathscr{E}(G,1)\rightarrow\mathscr{E}(G',1)$ between the unipotent series of $G$ and $G'$, in such a way that, for a unipotent representation $\pi$ of $G$

-- The representation $\theta(\pi)$ occurs in $\Theta(\pi)$.

-- If $\pi'$ belongs to $\Theta(\pi)$, then $\mathcal{O}_{\theta(\pi)}\preceq\mathcal{O}_{\pi'}$ 

 The last item above asserts that $\theta_{G,G'}(\pi)$ has the smallest unipotent support among irreducible representations in $\Theta_{G,G'}(\pi)$, it is in this sense ``minimal''.

 Before extending this to arbitrary representations, we recall the definition for unipotent representations. We work with symplectic orthogonal and unitary pairs independently.
 
\subsection{Unitary pairs}
Unipotent representations of the unitary groups $\U_n(q)$ are known to be indexed by partitions of $n$. Moreover, those belonging to the same Harish-Chandra series share a common $2$-\emph{core} and are therefore determined by their $2$-\emph{quotient} (of parameter one) \cite{AMR}. If we let $R_\mu$ be the representation of $\U_n(q)$ indexed by the partition $\mu$ of $n$, then the bijection issued from (\ref{Howlett-Lehrer}) relates $R_\mu$ to $\rho_{\mu(0),\mu(1)}$, where $(\mu(0),\mu(1))$ is the 2-quotient (of parameter 1) of $\mu$.  

Let $(G_m,G'_{m'})$ denote a unitary dual pair. According to Theorem 3.7 in \cite{AMR} the Howe correspondance relates the unipotent series $\mathscr{E}(G_m)_\varphi$ to the series $\mathscr{R}(G'_{m'})_{\varphi'}$, where $\varphi'$ is the first occurrence of $\varphi$. Moreover, since $\varphi\in\mathscr{E}(G_l)$ and $\varphi'\in\mathscr{E}(G'_{l'})$ are cuspidal and unipotent, the integers $l$ and $l'$ are triangular, i.e. $l=k(k+1)/2$ and $l'=k'(k'+1)/2$ for some $k$ and $k'$. From the discussion in the previous paragraph, the theta correspondence between the series above can be identified to a correspondence between $\mathscr{E}(W_r)$ and $\mathscr{R}(W_{r'})$, for suitable $r$ and $r'$. We can now define a bijection $\theta:\mathscr{E}(W_r)\rightarrow\mathscr{E}(W_{r'})$ issued from the one above :

%As for symplectic-orthogonal dual pairs we use Corollary 5.5 in \cite{HL} to identify the series above to the set of irreducible representations of two type B Weyl groups.  the bijection in Corollary 5.5 relates the representation indexed by the partition $\mu$ to its $2$-quotient (of parameter one) $(\mu(0),\mu(1))$. The theta correspondence between these series can then be identified to a correspondence $\Theta^\flat : \mathscr{E}(W_r)\rightarrow\mathscr{R}(W_{r'})$. Since representations of type B Weyl groups are parametrized by bipartitions, we can define a bijective correspondence $\underline{\theta}:\mathscr{E}(W_r)\rightarrow\mathscr{E}(W_{r'})$ as follows :
 
 -- $\theta(\lambda,\mu)=((r'-r)\cup\mu,\lambda)$, if $k$ is odd or zero.
 
 -- $\theta(\lambda,\mu)=(\mu,(r'-r)\cup\lambda)$, otherwise

Let $\mathcal{O}_\mu$ be the unipotent support of the unipotent character $R_\mu$. The closure order among unipotent supports agrees with the dominance order on the indexing partition \cite{Taylor}, i.e. $\mathcal{O}_\mu\preceq\mathcal{O}_\nu$ if and only if $\mu\leq\nu$. The definition of $\theta$ made above is done so that the unipotent support is at its smallest. Indeed, if $\theta(\mu)$ is the partition whose 2-quotient is $\theta(\mu(0),\mu(1))$, then for all  representations $R_{\mu'}$ in $\Theta(R_\mu)$ we have $\theta(\mu)\leq \mu'$. We obtain in this way a bijection $\theta_{G,G'}$ between the set of unipotent characters of $G$ and $G'$.
  
% Let now $\pi$ be an irreducible representation of $G_m$, and $\mathscr{E}(G_m,\rho)_L$ be its Harish-Chandra series. Let $\Xi_s^{G'_{m'}}(\pi)=\pi_\#\times \pi_{m-l}$. Restricting the bijection on theorem \ref{HC-UnipotentHC} we obtain 
%$$
%\Xi_s^{G'_{m'}}:\Theta^\flat_{m,m'}(\pi)\simeq \{\hat{\pi}_\#\}\times\Theta^\flat_{m-l,m'-l}(\pi_{m-l}).
%$$

%This motivates us to define 
%$$
%\underline{\theta}_{m,m'}:\mathscr{E}(G_m)\rightarrow\mathscr{E}(G'_{m'}),\hspace{0.5 cm} \underline{\theta}_{m,m'}(\pi)=\Xi_s^{-1}(\hat{\pi}_\#\times \underline{\theta}_{m-l,m'-l}(\pi_{m-l})-.
%$$ 

Let $\iota$ be an involution of sending a representation $\pi$ to its dual $\tilde{\pi}$. For the definition in the general case we use Theorem \ref{UnitaryReductionUnipotentCase}, which we can express as a commutative diagram
 
\begin{center} 
\begin{tikzcd}
 \mathscr{E}(G,(s)) \arrow[d, "\Theta_{G,G'}"] \arrow[r, "\Xi_{s}", "\sim" swap] 
                                                &  \mathscr{E}(G_\#\times G_{(1)},1) \arrow[d, "\iota\hspace{0.05cm}\otimes\hspace{0.05cm}\Theta_{G_{(1)},G'_{(1)}}"] \\
\mathscr{R}(G',(s')) \arrow[r,  "\Xi_{s'}", "\sim" swap] 
 												   & \mathscr{R}(G_\#\times G'_{(1)},1).
 \end{tikzcd}
\end{center}

%\begin{prop}\label{bijectionThetaThetaUnipotent}
%The irreducible representation $\pi'$ of $G'_{m'}$ belongs to $\Theta^\flat_{m,m'}(\pi)$, if and only if, $\pi'_\#=\hat{\pi}_\#$ and $\pi'_{m'-l}$ belongs to $\Theta^\flat_{m-l,m'-l}(\pi_{m-l})$. In particular, the map sending $\pi'\in\mathscr{E}(G'_{m'})$ to $\pi'_{m'-l}\in\mathscr{E}(G'_{m'-l})$ defines a bijection between $\Theta^\flat_{m,m'}(\pi)$ and $\Theta^\flat_{m-l,m'-l}(\pi_{m-l})$.
%\end{prop}
%This proposition tells us, in other words, that the Lusztig correpondence restricts to 
%$$
%\Xi_s^{G'_{m'}}:\Theta^\flat_{m,m'}(\pi)\simeq \{\hat{\pi}_\#\}\times\Theta^\flat_{m-l,m'-l}(\pi_{m-l}).
%$$

%Definition 7 in \cite{Epequin} introduced a partial order on the set $\Theta^\flat_{m,m'}(\pi)$, where $\pi$ is a unipotent representation of $G_m$. Using Proposition \ref{bijectionThetaThetaUnipotent}, we can extend this order to arbitrary irreducible representations. 

%\begin{defi}
%Let $\pi'$ and $\sigma'$ belong to $\Theta^\flat_{m,m'}(\pi)$. Then $\pi'\leq\sigma'$ if and only if $\pi'_{m'-l}\leq\sigma'_{m'-l}$. This defines a partial order in $\Theta^\flat_{m,m'}(\pi)$.
%\end{defi}

This motivates us to define $\theta_{G,G'}:\mathscr{E}(G)\rightarrow\mathscr{E}(G')$ by
$$
\Xi_{s'}(\theta_{G,G'}(\pi))=\tilde{\pi}_\#\times \theta_{G_{(1)},G'_{(1)}}(\pi_{(1)}).
$$
The representation $\theta_{G_{(1)},G'_{(1)}}(\pi_{(1)})$ has alredy been defined since $\pi_{(1)}$ is unipotent. We are therefore extending the definition of $\theta_{G,G'}$ to the whole set of irreducible representations, so that it is congruent with the Lusztig correspondence. It is only natural to ask if this extension also selects representations with smallest possible unipotent support. The answer is provided in a subsequent section.
 
%When $\pi$ is a unipotent representation, the order defined on $\Theta^\flat_{m,m'}(\pi)$ is not necessarily total. However, in Theorems 9 and 10 of \cite{Epequin} we were able to find a minimal and a maximal representations for this order. The same is true for arbitrary irreducible representations.

%\begin{thm}\label{ExtremalRepresentations}
%Let $\pi$ be an irreducible representation of $G_m$. There exists a unique minimal (resp. maximal) irreducible representation $\pi'_\mathrm{min}$ (resp. $\pi'_\mathrm{max}$) in $\Theta^\flat_{m,m'}(\pi)$.
%\end{thm}
%\begin{proof}
%Let $\sigma$ denote the unipotent representation $\pi_{m-l}$. According to Theorems 9 and 10 in \cite{Epequin}, there is an irreducible representation $\sigma'_\mathrm{min}$ (resp. $\sigma'_\mathrm{max}$) in $\Theta^\flat_{m-l,m'-l}(\sigma)$ verifying $\sigma'_\mathrm{min}\leq \sigma'$ (resp. $\sigma'\leq \sigma'_\mathrm{max}$) for all $\sigma'$ in $\Theta^\flat_{m-l,m'-l}(\sigma)$. Thanks to Proposition \ref{bijectionThetaThetaUnipotent}, we see that $\pi'_\mathrm{min}$ (resp. $\pi'_\mathrm{max}$) verifying $(\pi'_\mathrm{min})_{m'-l}=\sigma'_\mathrm{min}$ (resp. $(\pi'_\mathrm{max})_{m'-l}=\sigma'_\mathrm{max}$) is the desired minimal (resp. maximal) representation. 
%\end{proof} 
 
\subsection{Symplectic-orthogonal pairs}
Let $(G_m,G'_{m'})$ be a dual pair $(\Sp_{2m}(q),\Or_{2m'}^\epsilon)$. Again, according to Theorem 3.7 in \cite{AMR}, the Howe correspondence $\Theta_{m,m'}$ relates the unipotent Harish-Chandra series $\mathscr{E}(G_m)_\varphi$ to the series $\mathscr{E}(G'_{m'})_{\varphi'}$, where $\varphi'$ is the first occurrence of $\varphi$. Moreover, since $\varphi\in\mathscr{E}(G_l)$ and $\varphi'\in\mathscr{E}(G'_{l'})$ are cuspidal and unipotent, $l=k(k+1)$ and $l'=k^{'2}$ for some $k$ and $k'$. Again, thanks to (\ref{Howlett-Lehrer}), both these series can be identified to the set of representations of certain Weyl groups of type B. Hence, the correspondence between unipotent representations becomes a correspondence between the set of irreducible representations of a pair $(W_r,W_{r'})$ of type B Weyl groups, for suitable $r$ and $r'$. We now define the one-to-one correspondence $\theta: \Irr(W_r)\rightarrow\Irr(W_{r'})$ as follows :

-- $\theta(\lambda,\mu)=(\lambda,(r'-r)\cup\mu)$, when $k =\frac{1}{2}( \epsilon-1)$ mod $2$. 

-- $\theta(\lambda,\mu)=((r'-r)\cup\lambda,\mu)$, otherwise.

Using the Springer correspondence to calculate the unipotent support, in \cite{Epequin} we are able to prove that the choice above is made so that the support of $\theta(\pi)$ it the smallest in $\Theta(\pi)$ for unipotent $\pi$. We obtain in this way a bijection $\theta_{G,G'}$ between the set of unipotent characters of $G$ and $G'$.

%Let $\pi$ be an irreducible representation of $G_m$. We start by using Corollary \ref{CorollaryHC-HCUnipotent} [STATE SO ANALOGUE] to reduce the set $\Theta^\flat_{m,m'}(\pi)$ to a set of unipotent representations.

%\begin{prop}\label{SOThetaThetaUnipotent}
%The representation $\pi'$ of $G'_{m'}$ belongs to $\Theta^\flat_{m,m'}(\pi)$, if and only if, $\pi'_\#=\hat{\pi}_\#$, $\pi'_{(1)}=\hat{\pi}_{(1)}$ and $\pi'_{m'-l}$ belongs to $\Theta^\flat_{m-l,m'-l}(\pi_{m-l})$. In particular, the map sending $\pi'\in\mathscr{E}(G'_{m'})$ to $\pi'_{m'-l}\in\mathscr{E}(G'_{m'-l})$ defines a bijection between $\Theta^\flat_{m,m'}(\pi)$ and $\Theta^\flat_{m-l,m'-l}(\pi_{m-l})$.
%\end{prop}

Again, Let $\iota$ be an involution of sending a representation to its dual. For the general case, we use Theorem \ref{Symplectic-OrthogonalReductionUnipotentCase}, which again we express as a commutative diagram 

\begin{center} 
\begin{tikzcd}
 \mathscr{E}(G,(s)) \arrow[d, "\Theta_{G,G'}"] \arrow[r, "\Xi_{s}", "\sim" swap] 
                                                &  \mathscr{E}(G_\#\times G_{(-1)}\times G_{(1)},1) \arrow[d, "\iota\hspace{0.05cm}\otimes\hspace{0.05cm}\iota\hspace{0.05cm}\otimes\hspace{0.05cm}\Theta_{G_{(1)},G'_{(1)}}"] \\
\mathscr{R}(G',(s')) \arrow[r,  "\Xi_{s'}", "\sim" swap] 
 												   & \mathscr{R}(G_\#\times G_{(-1)}\times G'_{(1)},1).
 \end{tikzcd}
\end{center}

%This proposition tells us, in other words, that the Lusztig correspondence restricts to 
%$$
%\Xi_s^{G'_{m'}}:\Theta^\flat_{m,m'}(\pi)\simeq \{\hat{\pi}_\#\}\times\{\hat{\pi}_{(1)}\}\times\Theta^\flat_{m-l,m'-l}(\pi_{m-l}).
%$$

This moves us to define $\theta_{G,G'}(\pi)$ to be the representation such that
$$
\Xi_s(\theta_{G,G'}(\pi))=\tilde{\pi}_\#\otimes\tilde{\pi}_{(-1)}\otimes\theta_{G_{(1)},G'_{(1)}}(\pi_{(1)}).
$$
We stress that the representation $\theta_{G_{(1)},G'_{(1)}}(\pi_{(1)})$ has already been defined since $\pi_{(1)}$ is unipotent. We are therefore extending the definition of $\theta$ from unipotent to arbitrary representations making sure it is compatible with the Lusztig correspondence. Again, it seems reasonable to ask if this extension also selects a representation with smallest unipotent support. We provide the answer in the following section.

\subsection{Lusztig correspondence and unipotent support}

Assume $G$ is a connected reductive group. Let $P$ be a parabolic subgroup of $G$ with Levi decomposition $P=LU$. Following \cite{Spaltenstein} we introduce an induction functor on unipotent classes from $L$ to $G$ as follows : for a unipotent class $\mathcal{O}$ in $L$, there exists a unique unipotent class $\tilde{\mathcal{O}}$ of $G$ such that $\tilde{\mathcal{O}}\cap \mathcal{O}U$ is dense in $\mathcal{O}U$. We say that $\tilde{\mathcal{O}}$ is the class obtained inducing $\mathcal{O}$ from $L$ to $G$, and we write $\tilde{\mathcal{O}}=\Ind_L^G(\mathcal{O})$. This definition does not depend on the parabolic $P$ containing $L$. Moreover, according to Proposition II.3.2 in \cite{Spaltenstein}\begin{align}\label{InductionClosure}
\overline{\Ind_L^G(\mathcal{O})}=\bigcup_{x\in G}x\overline{\mathcal{O}U}x^{-1}
\end{align}

Let $s$ be a rational semisimple element of $\mathbf{G^*}$, and let $G(s)$ be the Langlands dual of $C_{G^*}(s)$. Since the unipotent series of these two groups can be identified, we have a Lusztig bijection between the series of $G$ defined by $s$ and the unipotent series of $G(s)$. Denote by $\rho_u$ the unipotent representation of $G(s)$ corresponding to $\rho$ in $\mathscr{E}(G,(s))$. If the group $C_{G^*}(s)$ can be identified to a Levi subgroup of $G^*$ then, according to Proposition 4.1 in \cite{GM}
\begin{align*}
\mathcal{O}_{\rho} = \Ind_{G(s)}^G(\rho_u),
\end{align*}
i.e. the unipotent supports of corresponding characters are related by the induction of classes defined above.

For both symplectic-orthogonal and unitary dual pairs $(G,G')$, we have first defined the bijection $\theta_{G,G'}$ on the set of unipotent representations combinatorially. We have then used the Lusztig correspondence to extend this definition to all irreducible representations. In the unipotent case the definition was made so as to minimize the unipotent support. If we aim at proving that this also holds for arbitrary representations, we must study the effect of the Lusztig correspondence on the unipotent support. The following result addresses this issue. As in (\ref{LusztigBijection}), we consider the Lusztig correspondence between the series $\mathscr{E}(G,(s))$ and $\mathscr{E}(C_{G^*}(s),1)$. %We let once again $(G,G')$ be a type I dual pair.

\begin{thm}\label{IncreasingLusztig}
Let $\pi$, $\tau$ belong to $\mathscr{E}(G,(s))$, and $\pi_u$, $\tau_u$ denote the corresponding unipotent representations of $C_{G^*}(s)$. If $\mathcal{O}_{\pi_u}\preceq\mathcal{O}_{\tau_u}$ then $\mathcal{O}_{\pi}\preceq\mathcal{O}_{\tau}$. In particular, for $\pi$ in $\mathscr{E}(G)$, the subrepresentation $\theta_{G,G'}(\pi)$ of $\Theta_{G,G'}(\pi)$ has the smallest unipotent support.
\end{thm}
\begin{proof}
We consider first the case where the centralizer of $s$ in $G^*$ is a Levi subgroup of $G^*$. From the discussion of centralizers in Section \ref{CentralizersSection}, this is the case for type A groups.

Let $L$ be a Levi subgroup contained in the parabolic $P=LU$. Take two unipotent classes $\mathcal{O}$ and $\mathcal{O}'$ in $L$, such that $\mathcal{O}'$ is contained in the closure $\overline{\mathcal{O}}$. This implies that $\mathcal{O}'U$ is contained in $\overline{\mathcal{O}U}$. Hence, due to (\ref{InductionClosure}) above :
$$
\overline{\Ind_L^G(\mathcal{O}')}=\bigcup_{x\in G}x\overline{\mathcal{O}'U}x^{-1}\subset\bigcup_{x\in G}x\overline{\mathcal{O}U}x^{-1}=\overline{\Ind_L^G(\mathcal{O})},
$$ 
whence $\Ind_L^G(\mathcal{O}')\preceq\Ind_L^G(\mathcal{O})$. Since in this case, as discussed above, the supports of $\pi$, and $\tau$ are obtained inducing those of $\pi_u$, and $\tau_u$ respectively, the result holds.

In the general case, the unipotent support of characters corresponding by the Lusztig bijection are related by the generalized induction defined in \cite{Spaltenstein}; as proven in Proposition 4.5 of \cite{GM}. The result follows from the fact that this induction is increasing (Remarque III.12.4.2 in \cite{Spaltenstein}). 
\end{proof}

%% file: UnipotentSupportRank.tex
\section{Unipotent support and rank}\label{SupportRank}
  
In a recent paper \cite{Howe-Gurevich}, Howe and Gurevich presented the notion of rank for representations of finite symplectic groups. This conduced to the introduction of the \emph{eta} correspondence, defined by the property of ``having maximal rank''; as explained below. %For the sake of completeness we recall the definition of rank and of the eta correspondence. Let $(G,G')$ be a dual pair consisting of one even-orthogonal group, and one symplectic group, and $\Theta$ be their Howe correspondence. For $\pi$ in $\mathscr{E}(G)$, the irreducible representation $\eta(\pi)$ of $G'$ is defined as the unique constituent of $\Theta(\pi)$ that has maximal rank. 

Let $\mathscr{L}_{n}(q)$ the group of symmetric matrices of size $n$ with coefficients in $\mathbb{F}_q$; or equivalently, of symmetric bilinear forms over a $\mathbb{F}_q$-vector space of dimension $n$ . Consider the \emph{Siegel parabolic} $P$ of $\Sp_{2n}(q)$, with Levi decomposition $P=LN$, where $L\simeq\GL_n(q)$ and 
$$
N = \left\{  \left(\begin{matrix}
1 & A\\
0 & 1
\end{matrix}\right) \mbox{ : } A \mbox{ belongs to } \mathscr{L}_n(q)\right\}.
$$
We identify the group $N$ to the group $\mathscr{L}_{n}(q)$. Being abelian, all its irreducible representations are one-dimensional. Moreover, fixing a non-trivial additive character $\psi$ of $\mathbb{F}_q$, we can define a bijection between $N$ and $\mathscr{E}(N)$ relating a symmetric $A$ to $\psi_A$; the latter being defined by $\psi_A(B)=\psi(\tr(AB))$, for $B$ symmetric.  

Let $\rho$ be a representation of $\Sp_{2n}(q)$. The restriction of $\rho$ to $N$ decomposes as a weighted sum of representations $\psi_A$, for $A$ symmetric. The orbits of the action of $L$ on $N$ can be identified with the orbits of $\GL_n(q)$ on $\mathscr{L}_{n}(q)$, i.e. with equivalence classes of symmetric matrices. The coefficients on the sum above are therefore constant on these classes. The first major invariant of a symmetric bilinear form is its rank. It is well known that, over
finite fields of odd characteristic, there are just two isomorphism classes of symmetric bilinear forms of a given rank $r$. We denote by $\mathscr{O}_r^+$ and $\mathscr{O}_r^-$, the two equivalence classes of symmetric matrices of rank $r$. The restriction of $\rho$ to $N$ decomposes as :
$$
\rho\rvert_N = \sum_{r=0}^n\sum_\pm m_{r\pm} \sum_{A\in\mathscr{O}_r^\pm}\psi_A 
$$
\begin{defi}
-- The \emph{rank} of the character $\psi_A$ is defined as the rank of $A$.

-- The \emph{rank} of $\rho$, denoted by $\rk(\rho)$, is defined as the greatest $k$ such that the restriction to $N$ contains characters of rank $k$, but of no higher rank.
\end{defi}

Consider the dual pair $(G,G') = (O^\pm_{2k}, Sp_{2n'})$ where $2k \leq n'$, i.e., the dual pair is in stable range. Howe and Gurevich show that for $\rho$ in $\mathscr{E}(O^\pm_{2k}(q))$, there is a unique irreducible representation in $\Theta_{G,G'}(\rho)$ of rank $2k$; all other constituents having smaller rank. This gives rise to a mapping called the \emph{eta} correspondence :
$$
\eta : \mathscr{E}(O_{2k}^\pm(q)) \rightarrow \mathscr{E}(Sp_{2n'}(q)).
$$
We now have two one-to-one correspondences $\theta$ and $\eta$ between the set of irreducible representations $\mathscr{E}(G)$ of $G$, and the corresponding set $\mathscr{E}(G')$ of $G'$. They are defined in different ways. The former chooses a subrepresentation of $\Theta$ with smallest unipotent support, whereas the latter selects one with greatest rank. In \cite{Pan3}, Shu-Yen Pan shows that these two ``theta relations'' agree on their common domain of definition (the stable range). This amounts to say that, among the irreducible constituents of $\Theta(\pi)$, the representation with the smallest unipotent support is the one having the greatest rank. This points out to an inverse relation between these two features. We could ask if this also holds for all representations. This is indeed the case for unipotent representations, the proof is based on results yet to be published by Shu Yen Pan (he has managed to calculate the rank of a unipotent representation in terms of its associated symbol). The general case still needs to be settled.

\begin{conj}\label{InverseRelation}
Let $\pi$ and $\tau$ be two representations of $G$. If $\mathcal{O}_\pi \preceq\mathcal{O}_\tau$ then $\rk(\tau)\leq \rk(\pi)$.
\end{conj}

The statement in the previous theorem is the best we can hope for. That is to say, we cannot ask the reverse implication to hold as well. The unipotent support is a geometrical object attached to a representation, whereas its rank is just a number. Since the closure order is a partial order, the reverse implication would tell us that the rank of the latter determines the former. 

%% file: WeakHarishChandra.tex
\section{Howe correspondence and weak Harish-Chandra theory}

\subsection{Weak Harish-Chandra theory}
Let $G = G_n$ a type I group. For an integer $0\leq r\leq n$, the \emph{pure standard} Levi subgroup $M_{r,n-r}$ of $G$ is a subgroup of the form
$$
M_{r,n-r} = G_{n-r}\times \GL_1 \times \ldots\times \GL_1
$$
where the linear group appears $r$ times. A Levi subgroup of $G$ is called \emph{pure} if it is conjugated to a pure standard Levi.

By the very definition, the set of pure Levi subgroups is stable by $G$-conjugation. It can also be proven \cite[Proposition 2.2]{GHJ} that if $M$ and $M'$ are pure Levi then the intersection $\prescript{x}{}{M}\cap M'$ is pure as well. This property is crucial in showing  that Harish-Chandra philosophy holds when we focus on set of pure Levi subgroups. 

Let $\pi$ be a representation of $G$, we say that $\pi$ is \emph{weakly cuspidal} if the parabolic restriction of $\pi$ to a proper pure Levi subgroup is trivial, i.e. $\prescript{*}{}{R}_M^G(\pi) = 0$ for every pure Levi subgroup $M$. A pair $(M,\pi)$, where $M$ is a pure Levi subgroup and $\pi$ is a weakly cuspidal representation of $M$, is called a \emph{weakly cuspidal pair}.  As in the usual cuspidal setting, these pairs provide a partition of the set of irreducible representations of $G$. Indeed, defining the \emph{weak Harish-Chandra series} corresponding by $(M,\pi)$, as the set of irreducible subrepresentations of the parabolic induced $R_M^G(\pi)$, we have.

\begin{prop}\cite[Proposition 2.3]{GHJ}
\begin{itemize}
\item[a)] The weak Harish-Chandra series partition the set of (isomorphism classes of) irreducible $G$-representations.
\item[b)] Every weak Harish-Chandra series is contained in some usual Harish-Chandra series.
\end{itemize}
\end{prop}

Item (b) implies that every usual Harish-Chandra series is partitioned into weak Harish-Chandra series, and hence shows that weak series refine the usual theory. These definitions are made for characteristic zero representations. However, the same applies to non-defining prime characteristic. 

In the non-zero characteristic case, weak Harish-Chandra prove better suited for studying unipotent representations. Let $l$ be a prime different from $p$, and $G_n$ denote the unitary group $U_{2n}(q)$ or $U_{2n+1}(q)$. From the work of Geck \cite{Geck2} we know that irreducible unipotent representations in characteristic $l$ of $G_n$ are (just as in trivial characteristic) labelled by partitions of $n$. Calling $\pi_\mu$ the unipotent representation corresponding to the partition $\mu$ we get a result analogous to \cite[Appendice, proposition p. 224]{FS}. 

\begin{prop}\label{CoreWeaKHC}
If the unipotent representations $\pi_\mu$ and $\pi_\nu$ of $G_n$ lie in the same weak Harish-Chandra series, then $\mu$ and $\nu$ have the same 2-core.
\end{prop}
  
This result was originally part of a series of conjectures stated in \cite{GHJ}, the main of which (Conjecture 5.7) is now a theorem \cite{DVV}. 

\subsection{Orbits and stabilizers on a quotient space}
Let $W_1$ be a symplectic space spanned by $\{ e_1,\ldots, e_{n_1},e'_{n_1},\ldots,e'_1\}$, and $W_2$ spanned by $\{f_1,\ldots, f_{n_2},f'_{n_2},\ldots,f'_1\}$; let $W_n=W_1\operp W_2$ be their orthogonal sum. Using these basis, we identify symplectic transformations with symplectic matrices.

Let $\mathfrak{X}_k(W)$ denote the set of isotropic $k$-dimensional subspaces of a symplectic space $W$. Fix the maximal isotropic space $X_n$ spanned by $\{e_1,\ldots,e_{n_1},f_1,\ldots,f_{n_2}\}$, and let $P_n$ be the parabolic subgroup of $\Sp_{2n}$, formed by those matrices stabilizing $X_n$. Using this Lagrangian we identify the quotient $\Sp_{2n}/P_n$ with the set  $\mathfrak{X}_n(W_n)$ of maximal isotropic subpaces of $W_n$.

Basic linear algebra shows that the set of maximal isotropic subspaces of $W_n$ can be identified to the the set of triplets $(U_1, U_2,\phi)$ where $U_1$ belongs to $\mathfrak{X}_{d_1}(W_{n_1})$, $U_2$ belongs to $\mathfrak{X}_{d_2}(W_{n_2})$, $\phi : U_1^\perp/U_1\rightarrow U_2^\perp/U_2$ is an isomorphism, and $d_1-d_2 = n_1-n_2$. Moreover, the action of $(x_1,x_2)$ in $\Sp_{2n_1}\times\Sp_{2n_2}$ on $\mathfrak{X}_n(W_n)$ corresponds to the following action  
\begin{align}\label{ActionTriplets}
(x_1,x_2):(U_1, U_2,\phi) \mapsto (x_1U_1,x_2U_2,x_2\circ\phi\circ x_1^{-1}).
\end{align}
on the set of triplets. 

Let $i = 1,\ldots,\min\{n_1,n_2\}$, and suppose that $n_2\leq n_1$. Let $K_n$ be a matrix of size $n$ with $1$'s on the antidiagonal and $0$'s elsewhere. Finally, let
\begin{comment}
$$
V_i = \left[\begin{matrix}
1 & T_i\\
0   & 1 
\end{matrix}\right] 
\left[\begin{matrix}
0 & K_n\\
-K_n  & 0 
\end{matrix}\right]
\mbox{, where }
T_i =\left[\begin{array}{c|c|c}
 & 0 & 0\\
 \hline
\begin{matrix}
      & 0_{n_2-i} \\
K_i & 
\end{matrix} & 0 & 0\\ 
\hline
 & \begin{matrix}
      & 0_{n_2-i} \\
K_i & 
\end{matrix} & 
\end{array}\right]
$$
\end{comment}
$$
V_i = \left[\begin{matrix}
1 & 0\\
T_i & 1 
\end{matrix}\right] 
\mbox{, where }
T_i =\left[\begin{array}{c|c|c}
& \begin{matrix}
      & K_i  \\
0_{n_2-i} & 
\end{matrix} & \\ 
\hline
& & \begin{matrix}
      & K_i  \\
0_{n_2-i}  
\end{matrix} \\
\hline
0 & &  
\end{array}\right]
$$

Proposition 3.4 in \cite{Kudla} asserts that the different $V_i P_n$ for $i=1,\ldots,n_2$ form a set of representatives for the action of $\Sp_{2n_1}\times\Sp_{2n_2}$ on $\Sp_{2n}/P_n$.  We are interested in calculating their stabilizers. 

The coset $V_i P_n$ corresponds to the maximal isotropic subspace $V_i X_n$. The isotropic spaces in the triplet $(U_1,U_2,\phi)$ corresponding to the latter are
\begin{align*}
U_1 = V_iX_n\cap W_1 \mbox{, and } U_2 = V_iX_n\cap W_2. %= \langle e'_{n_1-i},\ldots,e'_1\rangle \\
%U_2 = V_iX_n\cap W_2 = \langle f'_{n_2-i},\ldots,f'_1\rangle.
\end{align*}

An easy calculation shows that 
\begin{align*}%\label{OrthogonalQuotient}
U_1^\perp/U_1=\langle e_{n_1-i+1},\ldots,e_{n_1},e'_{n_1},\ldots,e'_{n_1-i+1}\rangle\\
U_2^\perp/U_2=\langle f_{n_2-i+1},\ldots,f_{n_2},f'_{n_2},\ldots,f'_{n_2-i+1}\rangle.
\end{align*}

Using coordinates in these basis, we define the isomorphism
$$
\phi : U_1^\perp/U_1\rightarrow U_2^\perp/U_2 \mbox{, } \phi(z)=Kz,
$$ 
where $K$ is of size $2i$. Using the description (\ref{ActionTriplets}) of the action of $\Sp_{2n_1}\times\Sp_{2n_2}$ on the set of triplets, we see that $(x_1,x_2)$ belongs to the stabilizer $(\Sp_{2n_1}\times\Sp_{2n_2})^{V_iP_n}$ of $V_iP_n$, if and only if  

\begin{align}\label{EquivalenceStabilizer}
x_1U_1=U_1, \mbox{ }x_2U_2=U_2\mbox{, }x_2\phi = \phi x_1.
\end{align}

As before, let $P_k$, $k=1,\ldots,m$ be the stabilizer in $\Sp_{2m}$ of the totally isotropic space spanned by the first $k$ vectors of a symplectic base. The first two equalities on (\ref{EquivalenceStabilizer}) above tell us that  
$$
x_1\in P_{n_1-i}\subset \Sp_{2n_1}, \mbox{ and } x_2\in P_{n_2-i}\subset \Sp_{2n_2}.
$$
Elements $x$ in the parabolic $P_k$ factorize as a product $x=m(a,A)u$, for $a\in \GL_k$, $A\in\Sp_{2(m-k)}$, $u$ in the unipotent radical $N_k$ of $P_k$, and $ m(a,A)=\diag(a,A,K{}^ta^{-1}K)$. 

Hence, we can express $x_1$, $x_2$ as a product $x_1=m(a_1,A_1)u_1$, $x_2=m(a_2,A_2)u_2$, for suitable $a_1$, $a_2$, $A_1$, $A_2$, $u_1$, and $u_2$. The last equality on (\ref{EquivalenceStabilizer}) becomes the identity $A_2K=KA_1$ in $\Sp_{2i}$.  

The same discussion can be put forward for even-orthogonal groups, in this case the orbit representatives are
\begin{comment}
$$
V_i = \left[\begin{matrix}
1 & T_i\\
0   & 1 
\end{matrix}\right] 
\left[\begin{matrix}
0 & K_n\\
K_n  & 0 
\end{matrix}\right]
\mbox{, where }
T_i =\left[\begin{array}{c|c|c}
 & 0 & 0\\
 \hline
\begin{matrix}
      & 0_{n_2-i} \\
-K_i & 
\end{matrix} & 0 & 0\\ 
\hline
 & \begin{matrix}
      & 0_{n_2-i} \\
K_i & 
\end{matrix} & 
\end{array}\right].
$$ 
\end{comment}
$$
V_i = \left[\begin{matrix}
1 & 0\\
T_i & 1 
\end{matrix}\right] 
\mbox{, where }
T_i =\left[\begin{array}{c|c|c}
& \begin{matrix}
      & K_i  \\
0_{n_2-i} & 
\end{matrix} & \\ 
\hline
& & \begin{matrix}
      & -K_i  \\
0_{n_2-i}  
\end{matrix} \\
\hline
0 & &  
\end{array}\right].
$$

Let $G_m$ be $\Sp_{2m}$ or $\Or^\pm_{2m}$. We summarize the above results in the following proposition.  

\begin{prop}\label{OrbitStabilizer}
The matrices $V_iP_m$, for $i=1,\ldots,\min\{m_1,m_2\}$ form a set of representatives for the orbits of $G_{m_1}\times G_{m_2}$ in $G_{m}/P_m$. Moreover the stabilizer of $V_iP_m$ in $G_{m_1}\times G_{m_2}$ is the subgroup of $P_{m_1-i}\times P_{m_2-i}$ given by
$$
(G_{m_1}\times G_{m_2})^{V_iP_m}  =  \left\{
\begin{array}{lc}
&  a_k\in\GL_{m_k-i} \\
(m(a_1,A_1)u_1,m(a_2,A_2)u_2) : & u_k\in N_k \\
&  A_2=KA_1K
\end{array}
\right\}
$$
\end{prop}

\subsection{First occurrence of weakly cuspidal representations}
In this section we endeavour to study the effect the Howe correspondence has on weakly cuspidal representations in the case of first occurrence. We also comment on the relation between the correspondence and weak series. 

Let $G_m$ and $P_m$ be as in Proposition \ref{OrbitStabilizer}.

\begin{lem}\cite[Corollary 3.3]{Kudla}\label{KudlaCoinvariants}
If $(\omega,S)$ is a model for the Weil representation, then space of $G_m$--coinvariants $S_{G_m}$ is isomorphic to a submodule of $\Ind_{P_m}^{G_m}(1)$.
\end{lem}

\begin{thm}\label{WeakFirstOccurrence}
Let $(G_m,G'_{m'})$ be a type I dual pair, and $\pi$ be an irreducible weakly cuspidal representation of $G_m$, and let $m'(\pi)$ be its first occurrence index. 

1. If $m'<m'(\pi)$, then $\Theta_{m,m'}(\pi)$ is empty

2. The representation $\Theta_{m,m'(\pi)}(\pi)$ is irreducible and weakly cuspidal

3.  If $m'>m'(\pi)$, then none of the constituents of $\Theta_{m,m'}(\pi)$ is weakly cuspidal 
\end{thm}
\begin{proof}
To avoid the excessive use or apostrophes we swap the groups in the dual pair and consider the pair $(G'_{m'},G_m)$ instead. 

Consider a weakly cuspidal representation $\pi'$ of $G'_{m'}$. The proof of the item 1 in the statement is the definition of the first occurrence index. The proof of the other two items can be divided in two independent parts : existence and uniqueness. The methods used in each are different.

\subsubsection{Existence}
We prove that if $\pi'\otimes\pi$ appears in the oscillator representation $\omega_{m',m(\pi')}$, then $\pi$ is weakly cuspidal.  

Assume that $\pi$ is not weakly cuspidal. In this case we can find an irreducible representation $\varphi_1$ of $G_{m(\pi')-1}$ such that $\pi\vert R_{M_1}^{G_{m}}(\chi_1\otimes\varphi_1)$.  Using Frobenius reciprocity we have
\begin{align*}
0 & \neq \langle\omega_{m',m},\pi'\otimes\pi\rangle \leq \langle\omega_{m',m},\pi'\otimes R_{M_1}^{G_{m}}(\chi_1\otimes \varphi_1)\rangle \\
  & = \langle 1\otimes\prescript{*}{}{R}_{M_1}^{G_{m}}(\omega_{m',m}),\pi'\otimes \chi_1\otimes \varphi_1\rangle.
\end{align*}
Proposition \ref{coinv-submod} implies that the last term above is bounded by
$$ 
\langle \omega_{m',m-1}\otimes 1_{\GL_1}, \pi'\otimes \chi_1\otimes \varphi_1 \rangle + \langle \mathbf{R}_{\GL_1}\otimes\omega_{m'-1,m-1}, \prescript{*}{}{R}_{M'_1}^{G'_{m'}}(\pi')\otimes \chi_1\otimes \varphi_1\rangle.
$$
Since $\pi'$ is weakly cuspidal the second term in this sum must be trivial. The first term yields $\varphi_1\vert \Theta_{m',m(\pi')-1}(\pi')$, which contradicts the minimality of $m(\pi')$.

\subsubsection{Uniqueness} In order to establish uniqueness we prove that at most one weakly cuspidal irreducible representation can appear in the union of $\Theta_{m',m}(\pi')$ for all non negative $m$.

Let $\pi_1$, and $\pi_2$ belong to $\Theta_{m',m_1}(\pi')$, and $\Theta_{m',m_2}(\pi')$ respectively. Following the arguments in Section 3 of \cite{Kudla}, it can be shown that the representation $\pi_1\otimes\tilde{\pi}_2$ is a constituent of the space of $G'_{m'}$--coinvariants $S_{G'_{m'}}$ of the Weil representation. Therefore, from Lemma \ref{KudlaCoinvariants} the representation $\pi_1\otimes\tilde{\pi}_2$ is a constituent of 
$$
\Ind_{P_{m}}^{G_{m}}(1)\vert_{G_{m_1}\times G_{m_2}}=\bigoplus_{i=1}^{\min\{m_1,m_2\}}\Ind_{(G_{m_1}\times G_{m_2})^{V_iP_{m}}}^{G_{m_1}\times G_{m_2}}(1),
$$
where the $V_iP_{m}$, for $i = 1,\ldots,\min\{m_1,m_2\}$, are the representatives of the orbits of $G_{m_1}\times G_{m_2}$ in $G_{m}/P_{m}$, described above. By transitivity
\begin{align*}
\Ind_{(G_{m_1}\times G_{m_2})^{V_iP_{m}}}^{G_{m_1}\times G_{m_2}}(1) 
					   = \Ind_{P_{m_1-i}\times P_{m_2-i}}^{G_{m_1}\times G_{m_2}}\circ\Ind_{(G_{m_1}\times G_{m_2})^{V_iP_{m}}}^{P_{m_1-i}\times P_{m_2-i}}(1).
\end{align*}					   
Moreover, from the explicit description of stabilizers already established, 
\begin{align*}					
\Ind_{(G_{m_1}\times G_{m_2})^{V_iP_{m}}}^{P_{m_1-i}\times P_{m_2-i}}(1) =    
1_{\GL_{m_1-i}}\otimes 1_{N_{m_1-i}}\otimes 1_{\GL_{m_2-i}}\otimes 1_{N_{m_2-i}} \otimes \mathbf{R}_{G_{i}}, 
\end{align*}
whence we deduce 
\begin{align*}
\Ind_{(G_{m_1}\times G_{m_2})^{V_iP_{m}}}^{G_{m_1}\times G_{m_2}}(1) = R_{M_{m_1-i}\times M_{m_2-i}}^{G_{m_1}\times G_{m_2}}(1_{\GL_{m_1-i}}\otimes 1_{\GL_{m_2-i}}\otimes \mathbf{R}_{G_{i}}).
\end{align*}
Since $1_{\GL_k}$ is a subrepresentation of $R_{\GL_1^k}^{\GL_k}(1)$, the character on the righthand side of the last equality is a constituent of  
\begin{align*}
& R_{M_{m_1-i}\times M_{m_2-i}}^{G_{m_1}\times G_{m_2}} (R_{\GL_1^{m_1-i}}^{\GL_{m_1-i}}(1)\otimes R_{\GL_1^{m_2-i}}^{\GL_{m_2-i}}(1)\otimes\mathbf{R}_{G_{i}})\\
& =  R_{M_{m_1-i}\times M_{m_2-i}}^{G_{m_1}\times G_{m_2}}\circ R_{M_{m_1-i,i}\times M_{m_2-i,i}}^{M_{m_1-i}\times M_{m_2-i}}(1_{\GL_1^{m_1-i}}\otimes 1_{\GL_1^{m_2-i}} \otimes \mathbf{R}_{G_{i}})\\
& =  R_{M_{m_1-i,i}\times M_{m_2-i,i}}^{G_{m_1}\times G_{m_2}}(1_{\GL_1^{m_1-i}}\otimes 1_{\GL_1^{m_2-i}} \otimes \mathbf{R}_{G_{i}})
\end{align*}
We must now distinguish the following two cases.

(a) If $m_1\neq m_2$ then, for all $i=1\ldots,\min\{m_1,m_2\}$, one of the Levi subgroups $M_{m_1-i,i}$ or $M_{m_2-i,i}$ is going to be proper in $G_{m_1}$ or $G_{m_2}$ respectively. In this case
$$
\langle \Ind_{P_{m}}^{G_{m}}(1)\vert_{G_{m_1}\times G_{m_2}}, \pi_1\otimes \tilde{\pi}_2\rangle=0
$$
since $\pi_1$ and $\pi_2$ are both weakly cuspidal. 

(b) If $m_1=m_2=k$, then for all $i=1\ldots,k$,
$$
\langle \Ind_{P_{m}}^{G_{m}}(1)\vert_{G_{m_1}\times G_{m_2}}, \pi_1\otimes \tilde{\pi}_2\rangle\leq 1.
$$

Indeed,

-- In case $0\leq i < k$, the Levi subgroups $M_{k-i,i}$ and $M_{k-i,i}$ are proper, and again 
$$
\langle \Ind_{(G_{k}\times G_{k})^{V_iP_{m}}}^{G_{k}\times G_{k}}(1), \pi_1\otimes \tilde{\pi}_2\rangle=0
$$

-- In case $i = k$, we get 
$$
\langle \Ind_{(G_{k}\times G_{k})^{V_iP_{m'}}}^{G_{k}\times G_{k}}(1), \pi_1\otimes \tilde{\pi}_2\rangle \leq \langle \mathbf{R}_{G_{k}},\pi_1\otimes\tilde{\pi}_2\rangle=\langle \pi_1,\pi_2 \rangle.
$$

Uniqueness follows from Items (a) and (b).  
\end{proof}

Now that we have proven that the Howe correspondence preserves weak cuspidality on the first occurrence, we can ask whether it also behaves nicely with respect to weak Harish-Chandra series. It turns out that this is indeed the case. The statement and proof of this result is very much the same as Theorem \ref{HoweHarish-Chandra} and it will be therefore ommited.